\documentclass{article}
\usepackage[utf8]{inputenc}
\usepackage[margin= 2.2 cm]{geometry}
\usepackage{amsmath}
\usepackage{amssymb}
\usepackage{amsthm}
\usepackage{graphicx}
\usepackage{amscd}
\usepackage{epic, eepic}
\usepackage{url}
\usepackage{color}
\usepackage[utf8]{inputenc} 
\usepackage{comment}
\usepackage{enumerate}   
\usepackage{todonotes}
\usepackage{graphicx}
\usepackage{epstopdf}
\usepackage{enumitem}
\usepackage{tikz}
\usepackage{subcaption}
\usepackage{setspace}
\usepackage[bookmarks=false,hidelinks]{hyperref}

\makeatletter
\renewenvironment{proof}[1][\proofname] {\par\pushQED{\qed}\normalfont\topsep6\p@\@plus6\p@\relax\trivlist\item[\hskip\labelsep\bfseries#1\@addpunct{.}]\ignorespaces}{\popQED\endtrivlist\@endpefalse}
\makeatother

\newtheorem{theorem}{\bf Theorem}[section]
\newtheorem{lemma}[theorem]{\bf Lemma}
\newtheorem{corollary}[theorem]{\bf Corollary}
\newtheorem{claim}[theorem]{\bf Claim}

\newtheorem{question}[theorem]{\bf Question}
\newtheorem{problem}[theorem]{\bf Problem}

\theoremstyle{plain}

\newtheorem{definition}[theorem]{\bf Definition}

\def\eps{\varepsilon}

\def\ex{\mathrm{ex}}
\def\Exx{\mathrm{Ex}}

\author{
Barnab\'as Janzer\thanks{Department of Mathematics, ETH Z\"urich,
Switzerland. Research supported by SNSF grant 200021\_228014
and by a fellowship at Magdalen College, University of Oxford. Email: \mbox{\textbf{barnabas.janzer@math.ethz.ch}}.}
\and
Oliver Janzer\thanks{Department of Pure Mathematics and Mathematical Statistics, University of Cambridge, United Kingdom. Research supported by a fellowship at Trinity College. Email: \textbf{oj224@cam.ac.uk}.}
\and
Abhishek Methuku\thanks{Department of Mathematics, University of Illinois at Urbana–Champaign, Urbana, IL, USA.
Email: \mbox{\textbf{
abhishekmethuku@gmail.com}}} 
\and
G\'abor Tardos\thanks{HUN-REN, Alfr\'ed R\'enyi Institute of Mathematics, Budapest, Hungary. Research supported by the National Research, Development and Innovation Office projects K-132696 and SNN-135643 and by the ERC Advanced Grants “ERMiD” and “GeoScape.” Email: \mbox{\textbf{tardos@renyi.hu}}}
}

\date{\vspace{-21pt}}

\title{Tight bounds for intersection-reverse sequences, edge-ordered graphs and applications}

\begin{document}

\maketitle

\begin{abstract}
    In 2006, Marcus and Tardos proved that if $A^1,\dots,A^n$ are cyclic orders on some subsets of a set of $n$ symbols such that the common elements of any two distinct orders $A^i$ and $A^j$ appear in reversed cyclic order in $A^i$ and $A^j$, then $\sum_{i} |A^i|=O(n^{3/2}\log n)$. This result is tight up to the logarithmic factor and has since become an important tool in Discrete Geometry. In this paper we improve this to the optimal bound $O(n^{3/2})$. In fact, we prove the following more general result.
    
    We show that if $A^1,\dots,A^n$ are linear orders on some subsets of a set of $n$ symbols such that no three symbols appear in the same order in any two distinct linear orders, then $\sum_{i} |A^i|=O(n^{3/2})$. Using this result, we resolve several open problems in Discrete Geometry and Extremal Graph Theory as follows. 
    \begin{enumerate}[label=(\roman*)]

        \item We prove that every $n$-vertex topological graph that does not contain a self-crossing four-cycle has $O(n^{3/2})$ edges. This resolves a problem of Marcus and Tardos from 2006.

        \item We show that $n$ pseudo-circles in the plane can be cut into $O(n^{3/2})$ pseudo-segments, which, in turn, implies new bounds on point-circle incidences and on other geometric problems. This goes back to a problem of Tamaki and Tokuyama from 1998 and improves several results in the area.

        \item We also prove that the edge-ordered Tur\'an number of the four-cycle $C_4^{1243}$ is $\Theta(n^{3/2})$. This gives the first example of an edge-ordered graph whose Tur\'an number is known to be $\Theta(n^{\alpha})$ for some $1<\alpha<2$, and answers a question of Gerbner, Methuku, Nagy, P\'alv\"olgyi, Tardos and Vizer. 
    \end{enumerate}

    Using different methods, we determine the largest possible extremal number that an edge-ordered forest of order chromatic number two can have. Kucheriya and Tardos showed that every such graph has extremal number at most $n2^{O(\sqrt{\log n})}$, and conjectured that this can be improved to $n(\log n)^{O(1)}$. We disprove their conjecture in a strong form by showing that for every $C>0$, there exists an edge-ordered tree of order chromatic number two whose extremal number is $\Omega(n 2^{C\sqrt{\log n}})$.
\end{abstract}

\section{Introduction}

A major line of research at the intersection of Discrete Geometry and Graph Theory is the study of classical problems of the latter field for graphs that have some geometric or algebraic structure. Recent highlights of this subject include a semi-algebraic version of Zarankiewicz's problem \cite{FPSSZ17}, which has applications to point-variety incidence bounds, and the resolution of the Erd\H os--Hajnal conjecture for string graphs \cite{Tom23}.

A topological graph is a representation of a graph in the plane such that the vertices of the graph are distinct points on the plane and the edges correspond to Jordan arcs joining the corresponding pairs of points. We assume that no edge passes through a vertex other than its endpoints and every two edges have a finite number of common interior points and they properly cross at each of these points. A geometric graph is a topological graph in which the edges are represented by straight line segments.

In the '60s, Avital and Hanani \cite{AH66} as well as Erd\H os and Perles (see~\cite{Pach99}) initiated, and later Kupitz \cite{Kup79} and many others continued the systematic study of extremal problems for geometric and topological graphs. In particular, they posed the following general question, which is a geometric analogue of the classical Tur\'an problem. Given a collection of forbidden geometric configurations, what is the maximum number of edges that an $n$-vertex topological graph can have without containing any of the forbidden configurations? For example, a famous problem in this subject asks for the maximum number of edges that a topological graph (or a geometric graph) on $n$ vertices can have without containing $k$ pairwise crossing edges. The conjectured answer is $O_k(n)$, but this is only known for $k=2$ (which is the trivial case corresponding to planar graphs), $k=3$ (see \cite{AAPPS97}, and for tighter bounds, see \cite{AT07}) and $k=4$ (see \cite{Ack09}). For $k>4$, the best known bound is $O_k(n(\log n)^{2\lceil\log k\rceil-4})$ for topological graphs in general~\cite{FPS24} and $O_k(n\log n)$ for geometric or $x$-monotone topological graphs (see \cite{Valtr98}). For surveys on Geometric Graph Theory, see e.g., \cite{Pach99,Pach13}.

Another well-studied instance of the general problem above is the following question, posed by Pach, Pinchasi, Tardos and T\'oth \cite{PPTT04} in 2004.

\begin{problem}[Pach--Pinchasi--Tardos--T\'oth \cite{PPTT04}] \label{prob:crossing-free}
    Let $K$ be a fixed abstract graph. What is the maximum number $\ex_{\textnormal{cr}}(n,K)$ of edges that a geometric graph on $n$ vertices can have if it contains no self-intersecting copy of $K$?
\end{problem}

Trivially, we have $\ex_{\textrm{cr}}(n,K)\geq \ex(n,K)$, and in particular, the answer is quadratic when $K$ is not bipartite. Moreover, if $K$ is not planar, then any copy of $K$ in a geometric (or topological) graph is self-intersecting, so $\ex_{\textrm{cr}}(n,K)=\ex(n,K)$. Hence, Problem \ref{prob:crossing-free} is the most interesting when $K$ is planar and bipartite. Pach, Pinchasi, Tardos and T\'oth \cite{PPTT04} showed that $\ex_{\textrm{cr}}(n,P_3)=\Theta(n\log n)$, where $P_3$ denotes the path with three edges. On the other hand, their best bound for the maximum number of edges in a \emph{topological} graph on $n$ vertices without a self-crossing $P_3$ is only $O(n^{3/2})$. Tardos \cite{Tar13} constructed, for each positive integer $k$, a geometric graph with a superlinear number of edges and no self-crossing path of length $k$.

One of the first results for topological graphs concerning Problem~\ref{prob:crossing-free} was due to Pinchasi and Radoi\v{c}i\'c~\cite{pinchasinumber}, who showed that the maximum number of edges in an $n$-vertex topological graph without a self-crossing four-cycle is $O(n^{8/5})$. Their proof relies on the study of intersection-reverse cyclic orders.

\begin{definition} \label{defn:int rev}
    Let $A^1,\dots,A^n$ be cyclic orders of subsets of a finite alphabet. We say that $A^i$ and $A^j$ are \emph{intersection-reverse} if their common elements appear in reverse order in the two cyclic orders. We say that $A^1,\dots,A^n$ are \emph{pairwise intersection-reverse} if for each $1\leq i<j\leq n$, $A^i$ and $A^j$ are intersection-reverse. 
\end{definition}

In what follows, for a cyclic or linear order $A$, we write $|A|$ for the number of elements ordered by $A$. For a vertex $v$ of a topological graph $G$, let $L_G(v)$ be the list of its neighbours, ordered cyclically counterclockwise according to the initial segment of the connecting edge. Pinchasi and Radoi\v{c}i\'c observed that if the lists $L_G(u)$ and $L_G(v)$ are not intersection-reverse for two distinct vertices $u$ and $v$ of the topological graph $G$, then $G$ contains a self-crossing four-cycle. Furthermore, they proved that if $A^1,\dots,A^n$ are pairwise intersection-reverse cyclic orders on subsets of size $d$ of a set of $n$ symbols, then $d=O(n^{3/5})$. Applying this result to the cyclic orders $L_G(u)$, $u\in V(G)$, they deduced that every $n$-vertex topological graph without a self-crossing four-cycle has $O(n^{8/5})$ edges.

The upper bound for the size of intersection-reverse cyclic orders (and hence also for the maximum number of edges in topological graphs without a self-crossing four-cycle) was greatly improved by Marcus and Tardos.

\begin{theorem}[Marcus--Tardos \cite{marcus2006intersection}] \label{thm:cyclic orders log n}
    Let $A^1,A^2,\dots,A^n$ be a collection of pairwise intersection-reverse cyclically ordered lists on subsets of a set of $n$ symbols. Then $\sum_{i=1}^n |A^i|=O(n^{3/2}\log n)$.
\end{theorem}

Note that this result is tight up to the logarithmic factor. Indeed, it is well-known that $\ex(n,C_4)=\Theta(n^{3/2})$ (see, e.g., \cite{furedi2013history}). This implies that there are examples of cyclically ordered lists $A^1,A^2,\dots,A^n$ with $\sum_{i=1}^n |A^i| = \Theta(n^{3/2})$ where no distinct $A^i$ and $A^j$ share more than one symbol (so they are trivially intersection-reverse). 

Consequently, Marcus and Tardos raised the following natural question.

\begin{question}[Marcus--Tardos \cite{marcus2006intersection}]
    Is the logarithmic factor needed in Theorem \ref{thm:cyclic orders log n}?
\end{question}

We answer this question by showing that the logarithmic factor can be completely removed.

\begin{theorem} \label{thm:main cyclic}
    Let $A^1,\dots,A^n$ be a collection of pairwise intersection-reverse cyclically ordered lists on some subsets of a set of $n$ symbols. Then $\sum_{i=1}^n |A^i|=O(n^{3/2})$.
\end{theorem}

In fact, we deduce Theorem \ref{thm:main cyclic} from a stronger result which concerns linear orders rather than cyclic orders.

\begin{theorem}\label{theorem_linear orders}
    There exists $C>0$ such that for any positive integer $n$, if $A^1,\dots,A^n$ are linear orders on some subsets of a set of $n$ symbols such that no three symbols appear in the same order in any two distinct linear orders $A^i$ and $A^j$, then $\sum_{i=1}^n |A^i|\leq Cn^{3/2}$.    
\end{theorem}

Given pairwise intersection-reverse cyclic orders $A^1,\dots,A^n$, we can turn each $A^i$ into a linear order $B^i$ by arbitrarily selecting the symbol where $B^i$ starts. Clearly, for any $i\neq j$, no three symbols appear in the same order in $B^i$ and $B^j$ (else the same three symbols are not reverse ordered in $A^i$ and $A^j$). Hence, by Theorem \ref{theorem_linear orders}, we have $\sum_{i=1}^n |A^i|=\sum_{i=1}^n |B^i|\leq Cn^{3/2}$. This shows that Theorem~\ref{theorem_linear orders} indeed implies Theorem \ref{thm:main cyclic}.

As demonstrated in \cite{marcus2006intersection}, Theorem \ref{thm:cyclic orders log n} has far-reaching consequences in Discrete Geometry. In Sections~\ref{subsec:selfcrossing}--\ref{subsec:tangencies}, we discuss some of the geometric consequences of our main results (Theorems \ref{thm:main cyclic} and \ref{theorem_linear orders}). In Section~\ref{edge-orderedapplication}, we present another application of our main results in the extremal theory of edge-ordered graphs.

\subsection{Self-crossing four-cycles in topological graphs}
\label{subsec:selfcrossing}

As mentioned above, using Theorem \ref{thm:cyclic orders log n}, Marcus and Tardos \cite{marcus2006intersection} showed the following.

\begin{theorem}[Marcus--Tardos \cite{marcus2006intersection}] \label{cor:MT C4}
   Every $n$-vertex topological graph without a self-crossing four-cycle has $O(n^{3/2}\log n)$ edges.
\end{theorem}

As we noted earlier, there exist $C_4$-free graphs on $n$ vertices with $\Theta(n^{3/2})$ edges, so the bound in Theorem~\ref{cor:MT C4} is tight up to the logarithmic factor. Using Theorem~\ref{thm:main cyclic}, we obtain a tight bound for the maximum number of edges in a topological graph without a self-crossing four-cycle.

\begin{corollary} \label{cor:crossing C4}
    Every $n$-vertex topological graph without a self-crossing four-cycle has $O(n^{3/2})$ edges.
\end{corollary}

Of course, this implies that every $n$-vertex geometric graph without a self-crossing four-cycle has $O(n^{3/2})$ edges.

\subsection{Number of incidences between pseudo-circles and points}
\label{subsec:incidences}

Perhaps the most important consequence of our result concerns cutting pseudo-circles and pseudo-parabolas into pseudo-segments, which are defined as follows.

\begin{definition}
\begin{enumerate}[label=$\mathrm{(\alph*)}$]
    \item  A collection of \emph{pseudo-circles} is a collection of simple closed Jordan curves, any two of which intersect at most twice, with proper crossings at each intersection.
\item  A collection of \emph{pseudo-parabolas} is a collection of graphs of continuous real functions defined on the entire real line such that any two intersect at most twice and they properly cross at these intersections.

\item     A collection of \emph{pseudo-segments} is a collections of curves, any two of which intersect at most once.

\item    For a collection $\mathcal{C}$ of pseudo-circles, the \emph{cutting number} $\tau(\mathcal{C})$ is the minimum number of cuts that transforms $\mathcal{C}$ into a collection of pseudo-segments. 
\end{enumerate}
\end{definition}

In 1998, Tamaki and Tokuyama~\cite{TT98} considered the problem of cutting pseudo-parabolas into pseudo-segments. Such results are very useful since pseudo-segments are much easier to work with compared to pseudo-parabolas and pseudo-circles. Furthermore, as we will see shortly, bounds on the cutting number of pseudo-circles directly translate into bounds on the number of incidences between pseudo-circles and points. Tamaki and Tokuyama~\cite{TT98} proved that $n$ pseudo-parabolas can be cut into $O(n^{5/3})$ pseudo-segments. Aronov and Sharir~\cite{AS02} showed that $n$ circles can be cut into $O(n^{3/2+o(1)})$ pseudo-segments. Agarwal, Aronov, Pach, Pollack and Sharir~\cite{agarwal2004lenses} used the result of Pinchasi and Radoi\v{c}i\'c about topological graphs without self-crossing four-cycles to prove that $n$ pseudo-parabolas or $n$ $x$-monotone pseudo-circles can be cut into $O(n^{8/5})$ pseudo-segments.

Marcus and Tardos \cite{marcus2006intersection} used Theorem \ref{thm:cyclic orders log n} to improve all of the above results.

\begin{theorem}[Marcus--Tardos \cite{marcus2006intersection}] \label{thm:MT pseudocircles}
    Let $\mathcal{C}$ be a collection of $n$ pseudo-parabolas or a collection of $n$ pseudo-circles. Then $\mathcal{C}$ can be cut into $O(n^{3/2}\log n)$ pseudo-segments.
\end{theorem}

A result of Agarwal et al. \cite{agarwal2004lenses} states that if $\mathcal{C}$ is a collection of $n$ pseudo-circles and $\mathcal{P}$ is a set of $m$ points, then the number of incidences between $\mathcal{C}$ and $\mathcal{P}$ satisfies $I(\mathcal{C},\mathcal{P})=O(m^{2/3}n^{2/3}+m+n+\tau(\mathcal{C}))$, where $\tau(\mathcal{C})$ is the cutting number of $\mathcal{C}$. In particular, Theorem \ref{thm:MT pseudocircles} implies that $I(\mathcal{C},P)=O(m^{2/3}n^{2/3}+m+n^{3/2}\log n)$.

Our Theorem \ref{thm:main cyclic} implies an improved bound for the cutting number of pseudo-parabolas and pseudo-circles (using the reduction from~\cite{marcus2006intersection}).

\begin{corollary} \label{cor:cutting number}
    Let $\mathcal{C}$ be a collection of $n$ pseudo-parabolas or a collection of $n$ pseudo-circles. Then $\mathcal{C}$ can be cut into $O(n^{3/2})$ pseudo-segments.
\end{corollary}

Corollary \ref{cor:cutting number}, in turn, implies (see~\cite{agarwal2004lenses,AS02,marcus2006intersection}) the following polylogarithmically improved bounds for the number of incidences between points and (pseudo-)circles.
\begin{corollary}\label{cor:incidences}
    If $\mathcal{C}$ is a collection of $n$ pseudo-circles and a $\mathcal{P}$ is a set of $m$ points in the plane, then the number of incidences is
    $$I(\mathcal{C},\mathcal{P})=O(m^{2/3}n^{2/3}+m+n^{3/2}).$$
    If, in addition, $\mathcal{C}$ is a collection of circles (not just pseudo-circles), then
    $$I(\mathcal{C},\mathcal{P})=O(m^{2/3}n^{2/3}+m^{6/11}n^{9/11}+m+n).$$
\end{corollary}

\subsection{Number of tangencies between families of disjoint curves}
\label{subsec:tangencies}

Another application of our results concerns the following problem, introduced in \cite{PST12} and attributed to Pinchasi and Ben-Dan. What is the maximum number of tangencies between the members of two families, each of which consists of pairwise disjoint curves, where two curves are called tangent if they intersect exactly once, and in that intersection point they do not cross, but touch each other? Pinchasi and Ben-Dan showed that any upper bound for the maximum number of edges in an $n$-vertex topological graph without a self-crossing four-cycle gives an upper bound for this problem as well (see \cite{KP23} for the proof of this reduction). Therefore, the result of Marcus and Tardos \cite{marcus2006intersection} on the extremal number of a self-crossing four-cycle (Theorem~\ref{cor:MT C4}) gave an upper bound of $O(n^{3/2}\log n)$ for the maximum number of tangencies between the members of two families of disjoint curves. On the other hand, Keszegh and P\'alv\"olgyi \cite{KP23} gave a lower bound of $\Omega(n^{4/3})$ for this problem and noted that ``even getting rid of the $\log n$ factor from the upper bound would be an interesting improvement''. Our Corollary \ref{cor:crossing C4} does exactly that, using the reduction observed by Pinchasi and Ben-Dan.

For more geometric consequences of our results, we refer the reader to \cite{marcus2006intersection}.

\subsection{Edge-ordered graphs}
\label{edge-orderedapplication}

In this subsection, we prove two results about the extremal numbers of edge-ordered graphs. While the first result is independent of the subject of intersection-reverse sequences, the second one will rely on our Theorem~\ref{theorem_linear orders}.

A systematic study of the extremal numbers of edge-ordered graphs was initiated by Gerbner, Methuku, Nagy, P\'alv\"olgyi, Tardos and Vizer \cite{GMNPTV23}, although specific problems of similar kind had been considered much earlier (see, e.g., \cite{CK71}).
Formally, an edge-ordered graph is a finite simple graph $G = (V,E)$ with a linear order on its edge set $E$. An isomorphism between edge-ordered graphs must respect the edge-order. A subgraph of an edge-ordered graph is itself an edge-ordered graph with the induced edge-order. We say that the edge-ordered graph $G$ contains another edge-ordered graph $H$ if $H$ is isomorphic to a subgraph of $G$. Otherwise we say that $G$ avoids $H$. For a positive integer $n$ and an edge-ordered graph $H$, let the extremal number of $H$ be the maximal number of edges in an edge-ordered graph on $n$ vertices that avoids $H$, and let this maximum be denoted by $\ex_<(n,H)$. Note that trivially we have $\ex_<(n,H)\geq \ex(n,H)$, where, with a slight abuse of notation, we used $H$ for both an edge-ordered graph and for its underlying unordered graph.

Gerbner et al. \cite{GMNPTV23} defined the so-called \emph{order chromatic number} $\chi_{\textrm{or}}(H)$ of an edge-ordered graph $H$ and used it to establish a version of the celebrated Erd\H os--Stone--Simonovits theorem for edge-ordered graphs. For non-empty edge-ordered graphs $H$, this parameter takes values in $\mathbb{Z}_{\geq 2}\cup \{\infty\}$. Their result then states that $\ex_<(n,H)=\binom{n}{2}$ if $\chi_{\textrm{or}}(H)=\infty$ and $\ex_<(n,H)=\left(1-\frac{1}{\chi_{\textrm{or}}(H)-1}+o(1)\right)\binom{n}{2}$ otherwise.

Analogously to the case of classical (unordered) extremal graph theory, this shows that the most interesting case is where $\chi_{\textrm{or}}(H)=2$, as in that case the above result does not determine the asymptotics of $\ex_<(n,H)$. A simple but important dichotomy in classical extremal graph theory states that if $H$ is a forest, then $\ex(n,H)=O(n)$, whereas if $H$ contains a cycle, then $\ex(n,H)=\Omega(n^{1+\eps})$ for some $\eps>0$ which can depend on $H$. Using the inequality $\ex_<(n,H)\geq \ex(n,H)$, it is still true that if $H$ is an edge-ordered graph containing a cycle, then $\ex_<(n,H)=\Omega(n^{1+\eps})$ for some $\eps>0$, but there exist edge-ordered forests with order chromatic number greater than two, so they can have extremal number $\Theta(n^2)$. The natural analogues of unordered forests in this context are therefore forests with order chromatic number two, and the extremal numbers of these graphs are of great interest. Gerbner et al. \cite{GMNPTV23} studied the extremal numbers of certain short edge-ordered paths with order chromatic number two and showed that the extremal number can be $\Omega(n\log n)$. In the other direction, Kucheriya and Tardos \cite{kucheriya2023characterization} proved that for every edge-ordered forest $H$ with order chromatic number two, we have $\ex_<(n,H)\leq n2^{O(\sqrt{\log n})}$, and conjectured that the stronger bound $\ex_<(n,H)\leq n(\log n)^{O(1)}$ should hold. This was the edge-ordered analogue of a similar conjecture on \emph{vertex-ordered graphs}: Pach and Tardos \cite{PT06} conjectured that vertex-ordered forests of interval chromatic number two have extremal number $n(\log n)^{O(1)}$. Recently, Pettie and Tardos~\cite{pettie2024refutation} refuted the conjecture of Pach and Tardos. In this paper, building upon their result, we also disprove the conjecture of Kucheriya and Tardos, and completely settle this problem, by showing the following result.

\begin{theorem} \label{thm:tree construction}
    For any $C>0$, there exists an edge-ordered tree $H$ with order chromatic number two such that $\ex_<(n,H)= \Omega(n2^{C\sqrt{\log n}})$.
\end{theorem}

This matches the upper bound $\ex_<(n,H)\leq n2^{O(\sqrt{\log n})}$ proved by Kucheriya and Tardos \cite{kucheriya2023characterization} for edge-ordered forests of order chromatic number two. Note that while our result completely settles the problem of how large the extremal number of an edge-ordered forest of order chromatic number two can be, the analogous question for vertex-ordered graphs is wide open. Indeed, while \cite{pettie2024refutation} proves the same lower bound for extremal functions of some vertex-ordered forests of interval chromatic number two, the upper bound analogous to the one proved for edge-ordered graphs in \cite{kucheriya2023characterization} is entirely missing (and seems to be hard) for vertex-ordered graphs. 

We know much less about the extremal numbers of edge-ordered graphs containing a cycle compared to that of edge-ordered forests. In fact, prior to our work, the only edge-ordered graphs $H$ for which the order of magnitude of $\ex_<(n,H)$ were known are some small forests and graphs with order chromatic number greater than two. In particular, there was no edge-ordered graph $H$ for which we knew that $\ex_<(n,H)=\Theta(n^{\alpha})$ for some $\alpha\in (1,2)$. A result of Gerbner et al. \cite{GMNPTV23} came close to showing the existence of such a graph. Let $C_4^{1243}$ be the four-cycle $abcd$ whose edges are ordered as $ab<bc<da<cd$. (We remark that there are two other edge-orderings of $C_4$, but they both have order chromatic number $\infty$, so their extremal number is $\binom{n}{2}$.) Gerbner et al. \cite{GMNPTV23} proved that $\ex_<(n,C_4^{1243})=O(n^{3/2}\log n)$, which comes close to the trivial lower bound $\Omega(n^{3/2})$. They asked whether $\ex_<(n,C_4^{1243})=\Theta(n^{3/2})$. We use our Theorem \ref{theorem_linear orders} to answer this question affirmatively.

\begin{theorem} \label{thm:ordered C4}
    $\ex_<(n,C_4^{1243})=\Theta(n^{3/2})$.
\end{theorem}

As mentioned above, this is the first edge-ordered graph of order chromatic number two which is not a forest and whose extremal number is known (up to a constant factor).

\medskip

\noindent \textbf{Organization of the paper.} In Section \ref{sec:linear orders}, we prove our main result (Theorem \ref{theorem_linear orders}) after giving a detailed outline of our proof. In Section~\ref{sec:edge-ordered graphs}, we prove Theorems~\ref{thm:tree construction} and \ref{thm:ordered C4}. Some concluding remarks and open problems are given in Section \ref{sec:concluding remarks}.

\section{Linear orders with restricted intersections} \label{sec:linear orders}

\subsection{Proof outline}
\label{subsec:proofoutline}

In this section we prove Theorem \ref{theorem_linear orders}. Before we turn to the formal proof, we outline our strategy and comment on some of the differences compared to the argument in \cite{marcus2006intersection}.

Throughout Section~\ref{sec:linear orders}, we will have a finite set of symbols $X$, and some linear orders $A^i$ defined on some subsets $X_i\subseteq X$. Equivalently, each $A^i$ is a non-repeating sequence in $X$ with image $X_i$, i.e., it is a permutation of $X_i$. With a slight abuse of notation, we will identify the sets $X_i$ with $A^i$, so we write $|A^i|$ for $|X_i|$, $|A^i\cap A^j|$ for $X_i\cap X_j$, and $a\in A^i$ for $a\in X_i$.

The starting idea behind the proof (which was already present in \cite{pinchasinumber}) is to define a quantity which can be easily bounded from below, but which can also be upper bounded efficiently provided that no two linear orders $A^i$ and $A^j$ have three common symbols appearing in the same order in the two linear orders. Roughly speaking, this quantity, denoted by $S$, is the sum of $\lambda_{a,b}$ over all pairs of distinct symbols $a$ and $b$, where $\lambda_{a,b}$ is the number of pairs $(A^i,A^j)$ which agree on the order of $a$ and $b$ minus the number of pairs $(A^i,A^j)$ which disagree on the order of $a$ and $b$. It is easy to show that $\lambda_{a,b}$ is close to being a perfect square, which implies that $\lambda_{a,b}$ cannot be very negative and hence the sum $S = \sum_{a,b} \lambda_{a,b}$ also cannot be very negative. This yields a lower bound on $S$.

Hence, our goal is to obtain a good enough upper bound on the sum $S$ using the fact that no two linear orders $A^i$ and $A^j$ have three common symbols appearing in the same order. To do this, it suffices to essentially `swap the order of summation', and argue that a typical pair $A^i$ and $A^j$ must disagree on the order of most pairs of symbols in their intersection and hence it must contribute very negatively to the sum $S = \sum_{a,b} \lambda_{a,b}$. Unfortunately, it turns out that there can be pairs $(A^i,A^j)$ whose contribution to $S$ is positive (though, importantly, it cannot be very large). On the other hand, if $A^i$ and $A^j$ have the stronger property that they disagree on the order of any two (rather than three) symbols (in which case, analogously to Definition \ref{defn:int rev}, we call $A^i$ and $A^j$ \emph{intersection-reverse}, or int-rev for short), then the contribution of $(A^i,A^j)$ to $S$ is entirely negative. 

A very important idea (versions of which were already used in \cite{pinchasinumber} and \cite{marcus2006intersection}) is to split each linear order $A^i$ into shorter subsequences and argue that many pairs of these short sequences $B$ and $B'$ are intersection-reverse. 
Any pair $(B,B')$ of these short sequences which are intersection-reverse will contribute roughly $-|B\cap B'|^2$ to the sum $S$, and the key task is to argue that the total contribution of short sequences which are intersection-reverse (i.e., $\sum_{(B,B') \textrm{ int-rev}} -|B\cap B'|^2$) will dominate the small positive contributions of the other pairs of short sequences. However, to do this, the previous papers~\cite{pinchasinumber,marcus2006intersection} needed to split the sequences $A^i$ into \emph{many} short subsequences. This means that, if each $A^i$ contains $\Theta(n^{1/2})$ symbols, then these short sequences will contain only $o(n^{1/2})$ symbols. So, typically, they will not even intersect at all, making a proof along these lines impossible. This is why in the previous papers \cite{pinchasinumber,marcus2006intersection}, the obtained upper bound for $\sum_i |A^i|$ is larger than the optimal bound $\Theta(n^{3/2})$.

In this paper, we will split each $A^i$ into only two subsequences $A^i_0$ and $A^i_1$ of roughly equal size, and our main innovation is a novel way to show that the (negative) contribution of the pairs of short sequences that are intersection-reverse to the sum $S$ is sufficiently large.
This new approach differs from the one in \cite{marcus2006intersection} at several places. One of the main differences is that while \cite{marcus2006intersection} lower bounds $\sum_{(B,B') \textrm{ int-rev}} |B\cap B'|^2$ by first lower bounding the linear quantity $\sum_{(B,B') \textrm{ int-rev}} |B\cap B'|$, we relate $\sum_{(B,B') \textrm{ int-rev}} |B\cap B'|^2$ to the quadratic expression $\sum_{i \not =j}|A_0^i\cap A_0^j||A_1^i\cap A_1^j|$, whose introduction is one of our key ideas (see Lemma~\ref{lemma_gainbound}). The main advantage of this new quadratic expression is that it can then be related to the sum $\sum_{i}|A_0^i||A_1^i|$ using a double counting argument (see Lemma \ref{lemma_cherrycounting}). Moreover, since $\sum_{i}|A_0^i||A_1^i| = \Omega(\sum_i |A^i|^2)$ (and this latter sum is large enough by Jensen's inequality), this yields the desired lower bound on $\sum_{(B,B') \textrm{ int-rev}} |B\cap B'|^2$, showing that indeed the total negative contribution of short intersection-reverse sequences to $S$ sufficiently dominates the small positive contributions of the other pairs of short sequences, resulting in the desired upper bound on $S$.

\subsection{Proof of Theorem \ref{theorem_linear orders}}

We now turn to the formal proof of Theorem \ref{theorem_linear orders}.
Given linear orders $A^1,\dots,A^n$ on some subsets of the set of symbols $\{a_1,\dots,a_{n'}\}$, we define the graph corresponding to these orders and symbols to be the bipartite graph with parts $\{A^1,\dots,A^n\}$ and $\{a_1,\dots,a_{n'}\}$ such that $a_i$ is joined to $A^j$ if and only if $a_i\in A^j$.

For a graph $G$ we write $\Delta(G)$ for the maximum degree over vertices in $G$ and $\delta(G)$ for the minimum degree. We say that a graph $G$ is $K$-almost-regular if $\Delta(G)\leq K\delta(G)$. The following lemma of Jiang and Seiver~\cite{jiang2012turan}, which is a modification of a result of Erdős and Simonovits~\cite{ErdosSimonovits}, will allow us to assume that the graph corresponding to our orders and symbols is $K$-almost-regular for some absolute constant $K$. Note that $e(G)$ denotes the number of edges in $G$.

\begin{lemma}[Jiang and Seiver~\cite{jiang2012turan}]\label{lemma_JiangSeiver}
    Let $\varepsilon, c$ be positive reals, where $\varepsilon<1$ and $c\geq 1$. Let $n$ be a positive integer that is sufficiently large as a function of $\varepsilon$. Let $G$ be an $n$-vertex graph with $e(G)\geq cn^{1+\varepsilon}$. Then, writing $K=20\cdot 2^{\frac{1}{\varepsilon^2}+1}$, $G$ contains a $K$-almost-regular subgraph $G'$ on $m\geq n^{\frac{\varepsilon}{2}\frac{1-\varepsilon}{1+\varepsilon}}$ vertices such that $e(G')\geq \frac{2c}{5}m^{1+\varepsilon}$.
\end{lemma}

Given $n$ sufficiently large, consider the (bipartite) graph $G$ corresponding to the linear orders $A^1,\dots,A^n$ and the $n$ symbols (as in the statement of Theorem~\ref{theorem_linear orders}), and apply Lemma~\ref{lemma_JiangSeiver}. The resulting (bipartite) subgraph $G'$ is the graph corresponding to some subsequences of some of the linear orders $A^i$, and a subset (of our original set of $n$ symbols) consisting of $n'$ symbols. Thus, to prove Theorem~\ref{theorem_linear orders}, it suffices to prove the following result.
	
\begin{theorem}\label{theorem_almostregular}
    For every $K>0$, there exists $C>0$ such that the following holds. Assume that $A^1,\dots,A^n$ are linear orders on some subsets of a set of $n'$ symbols such that no three symbols appear in the same order in any two distinct linear orders $A^i$ and $A^j$. Assume, furthermore, that the graph $G$ corresponding to these linear orders and symbols is $K$-almost-regular. Then $\sum_i |A^i|\leq Cn^{3/2}$.
\end{theorem}

Note that if the conditions in Theorem~\ref{theorem_almostregular} hold, then we have $\delta(G)n'\leq e(G) \leq \Delta(G)n$. Hence, $n'\leq Kn$ since $G$ is $K$-almost-regular. Similarly, $\delta(G)n \leq e(G) \leq \Delta(G)n'$, so we have $n\leq Kn'$. In the rest of this section, we prove Theorem~\ref{theorem_almostregular}.

Now we recall some setup (with some modifications we will need) from~\cite{marcus2006intersection}. In what follows, we will always assume that $A^1,\dots,A^n$ are linear orders such that no three symbols appear in the same order in any two distinct linear orders $A^i$, $A^j$. Moreover, we will often view these linear orders as non-repeating sequences, as mentioned earlier.
    For each $i\in[n]$, we divide the sequence $A^i$ into two subsequences $A^i_0, A^i_1$ of lengths $\left\lceil |A^i|/2\right\rceil$ and $\left\lfloor |A^i|/2\right\rfloor$ respectively (so that $A^i$ is the concatenation of $A^i_0$ and $A^i_1$). Given a sequence $B$ and two distinct symbols $a,b\in[n']$, we define
    \begin{equation*}
        f(B,a,b)=
        \begin{cases}
            0 & \text{if $a\not \in B$ or $b\not \in B$,}\\
            1 & \text{if $a$ precedes $b$ in $B$,}\\
            -1 & \text{if $b$ precedes $a$ in $B$}.
        \end{cases}
    \end{equation*}
    Moreover, given two sequences $B,B'$ and distinct symbols $a,b\in[n']$, we define
    \begin{equation*}
        f(B,B',a,b)=f(B,a,b)f(B',a,b).
    \end{equation*}
    In other words, $f(B,B',a,b)$ is $1$ if $a,b$ appear in the same order in $B$ and $B'$, and $-1$ if they appear in reverse order (and $0$ if $a$ and $b$ are not both contained in $B$ and $B'$).

Our proof relies on obtaining both a lower bound and an upper bound for the quantity
\begin{equation*}
    S=\sum_{i\not =j}\sum_{\epsilon\in\{0,1\}}\sum_{a\not =b} f(A^i_\epsilon, A^j_{\epsilon},a,b).
\end{equation*}

\subsubsection{Lower bounding $S$} 

First observe that $S$ cannot be very negative, because it is very close to being a sum of perfect squares.
\begin{lemma}[See also~\cite{marcus2006intersection}]\label{lemma_Slowerbound}
    $S\geq -\frac{1}{2}\sum_i|A^i|^2$.
\end{lemma}
\begin{proof}[Proof of Lemma~\ref{lemma_Slowerbound}]
    Note that
    \begin{align*}
        S&=\sum_{i\not =j}\sum_{\epsilon\in\{0,1\}}\sum_{a\not =b} f(A^i_\epsilon, A^j_{\epsilon},a,b)\\
        &=\sum_{\substack{a\not =b,\\\epsilon\in\{0,1\}}}\sum_{i\not =j} f(A^i_\epsilon,a,b)f(A^j_{\epsilon},a,b) \\ 
        &=\sum_{\substack{a\not =b,\\\epsilon\in\{0,1\}}}\left(\left(\sum_{i} f(A^i_\epsilon,a,b)\right)^2-\sum_{i} f(A^i_\epsilon,a,b)^2\right)\\
        &\geq \sum_{\substack{a\not =b,\\\epsilon\in\{0,1\}}}\left(-\sum_{i} f(A^i_\epsilon,a,b)^2\right)\\
        &=-\sum_{\substack{i,\\ \epsilon\in\{0,1\}}}2\binom{|A^i_{\epsilon}|}{2}\\
        &\geq-\frac{1}{2}\sum_i|A^i|^2,
    \end{align*}
    as desired.
\end{proof}

\subsubsection{Upper bounding $S$ in terms of the sum of squares of 
intersection sizes of intersection-reverse pairs}

    In the rest of the proof, we will obtain an upper bound on $S$, which will contradict the lower bound obtained in Lemma~\ref{lemma_Slowerbound}. Given two sequences $B$ and $B'$, we say that $B$, $B'$ are \emph{intersection-reverse}, or \emph{int-rev} for short, if their common elements appear in reverse order in the two sequences.
    
    Using Dilworth's theorem and similar reasoning as in \cite{pinchasinumber,marcus2006intersection}, we prove the following bound on $\sum_{a\not =b} f(A^i_{\epsilon}, A^j_{\epsilon}, a,b)$. Here, crucially, we obtain a small `gain' if $A_\epsilon^i$ and $A_\epsilon^j$ are intersection-reverse compared to the case when they are not.
    
\begin{lemma}\label{lemma_fbounds}
    Let $\epsilon \in \{0, 1\}$. For all pairs $A_\epsilon^i$, $A_\epsilon^j$ ($i\ne j$), we have 
    \begin{equation*}\label{equation_singular}
        \sum_{a\not =b} f(A^i_{\epsilon}, A^j_{\epsilon}, a,b)\leq |A^i_{\epsilon}\cap A^j_{\epsilon}|.
    \end{equation*}
    Moreover, for all intersection-reverse pairs $A^i_{\epsilon}$, $A^j_{\epsilon}$, we have
    \begin{equation*}\label{equation_reverse}
        \sum_{a\not =b} f(A^i_{\epsilon}, A^j_{\epsilon}, a,b)= |A^i_{\epsilon}\cap A^j_{\epsilon}|-|A^i_{\epsilon}\cap A^j_{\epsilon}|^2.
    \end{equation*}
\end{lemma}
\begin{proof}
Let $k=|A^i_{\epsilon}\cap A^j_{\epsilon}|$.
    (First, we make no assumption about $A^i_{\epsilon}, A^j_{\epsilon}$ being intersection-reverse.) Since no three symbols can appear in the same order in $A^i$, $A^j$, it follows that $A^i_{\epsilon}, A^j_{\epsilon}$ also do not contain three symbols in the same order. Hence, using (the dual of) Dilworth's theorem~\cite{dilworth1950decomposition} (applied to the poset defined on $A^i_{\epsilon}\cap A^j_{\epsilon}$ in which $a<b$ if $a$ precedes $b$ in both $A^i_{\epsilon}$ and $A^j_{\epsilon}$), we obtain that $A^i_{\epsilon}\cap A^j_{\epsilon}$ can be partitioned into two subsets, each of which are ordered reversely in $A^i_{\epsilon}$ and $A^j_{\epsilon}$. That is, $A^i_{\epsilon}\cap A^j_{\epsilon}=T_1\cup T_2$ such that for each $\ell\in\{1,2\}$ and distinct $a,b\in T_\ell$, we have $f(A^i_{\epsilon}, A^j_{\epsilon},a,b)=-1$. Hence, 
   \begin{align*}
       \sum_{a\not =b} f(A^i_{\epsilon}, A^j_{\epsilon}, a,b)      &\leq 2|T_1||T_2|-|T_1|(|T_1|-1)-|T_2|(|T_2|-1)\\
       &\leq |T_1|+|T_2|=k,
   \end{align*} 
   as claimed.

   On the other hand, if $A^i_{\epsilon}, A^j_{\epsilon}$ are intersection-reverse, then $f(A^i_{\epsilon}, A^j_{\epsilon},a,b)=-1$ for all distinct $a,b\in A^i_\epsilon\cap A^j_\epsilon$, and hence 
   $$\sum_{a\not =b} f(A^i_{\epsilon}, A^j_{\epsilon}, a,b)=-2\binom{k}{2}=k-k^2,$$
   as claimed. This proves the lemma.
\end{proof}

Applying Lemma~\ref{lemma_fbounds} for each pair $A^i_\epsilon, A^j_{\epsilon}$ (with $\epsilon \in \{0, 1\}$), and adding up the resulting bounds on $\sum_{a\not =b} f(A^i_{\epsilon}, A^j_{\epsilon}, a,b)$, we get
\begin{align}
    S&\leq \sum_{\substack{i\not=j, \epsilon}}|A^i_{\epsilon}\cap A^j_{\epsilon}|-\sum_{\substack{i\not=j,\epsilon:\\
    \textnormal{$A^i_{\epsilon}, A^j_{\epsilon}$ int-rev}}}|A^i_{\epsilon}\cap A^j_{\epsilon}|^2\nonumber\\
    &\le\sum_{\substack{i\not=j}}|A^i\cap A^j|-\sum_{\substack{i\not=j,\epsilon:\\
    \textnormal{$A^i_{\epsilon}, A^j_{\epsilon}$ int-rev}}}|A^i_{\epsilon}\cap A^j_{\epsilon}|^2\label{equation_Supperbound}.
\end{align}

\subsubsection{Relating the sum of squares to another quadratic expression}

To obtain the desired upper bound on $S$, we will need to lower bound the second sum of the right hand side of~\eqref{equation_Supperbound}. That is, we will have to show that intersection-reverse pairs tend to have large intersection size squared. This will be done using the following key lemma which relates this sum to another quadratic expression which is easier to bound (as shown by Lemma~\ref{lemma_cherrycounting}). 

\begin{lemma}\label{lemma_gainbound}
    If $i\not =j$, then
    \begin{equation*}
        \sum_{\epsilon:\textnormal{ $A^i_{\epsilon}, A^j_{\epsilon}$ int-rev}}|A^i_{\epsilon}\cap A^j_{\epsilon}|^2\geq 2|A^i_0\cap A^j_0||A^i_1\cap A^j_1|.
    \end{equation*}
\end{lemma}

The following claim is crucial to the proof of Lemma~\ref{lemma_gainbound}.
\begin{claim}\label{lemma_disjoint}
    If $i\not =j$ and the pair $A^i_{\epsilon}, A^j_{\epsilon}$ is not intersection-reverse, then $A^i_{1-\epsilon}$ and $A^j_{1-\epsilon}$ are disjoint.
\end{claim}
\begin{proof}
    If the pair $A^i_{\epsilon}, A^j_{\epsilon}$ is not intersection-reverse, then there are some symbols $a,b$ appearing in both of them such that $a$ precedes $b$ in both $A^i_{\epsilon}$ and $A^j_{\epsilon}$. If there is some symbol $c$ contained in both $A^i_{1-\epsilon}$ and $A^j_{1-\epsilon}$, then the elements $a,b,c$ must appear in the same order in both $A^i$ and $A^j$ (namely, they appear in the order $a,b,c$ if $\epsilon=0$ and in the order $c,a,b$ if $\epsilon=1$), contradicting the assumption that $A^i$, $A^j$ do not have three symbols appearing in the same order.
\end{proof}

\begin{proof}[Proof of Lemma~\ref{lemma_gainbound}]
If both of the pairs $A^i_0,A^j_0$ and $A^i_1, A^j_1$ are intersection-reverse, then
    \begin{equation*}
        \sum_{{\epsilon:\textnormal{ $A^i_{\epsilon}, A^j_{\epsilon}$ int-rev}}}|A^i_{\epsilon}\cap A^j_{\epsilon}|^2= |A^i_0\cap A^j_0|^2+|A^i_1\cap A^j_1|^2\geq 2|A^i_0\cap A^j_0||A^i_1\cap A^j_1|.
    \end{equation*}

    On the other hand, if one of the pairs $A^i_0,A^j_0$ and $A^i_1, A^j_1$ is not intersection-reverse, then the sequences in the other pair are disjoint by Claim~\ref{lemma_disjoint}, so $2|A^i_0\cap A^j_0||A^i_1\cap A^j_1|=0$ and the lemma holds trivially.
\end{proof}

We prove the following lemma using a double counting argument.

\begin{lemma}\label{lemma_cherrycounting}
    Let $\sum_i|A^i|=M$. If the graph corresponding to our $n$ linear orders $A^1,\dots,A^n$ and $n'$ symbols is $K$-almost-regular, $C=C(K)$ is sufficiently large, and $M\geq Cn^{3/2}$, then
    $$\sum_{i \not =j}|A_0^i\cap A_0^j||A_1^i\cap A_1^j|\geq \frac{M^4}{50K^2n^4}.$$
\end{lemma}
\begin{proof}
For any two symbols $a,b$, let $s_{a,b}=|\{i: a\in A_0^i, b\in A_1^i\}|$. Then,

\begin{equation}
\label{lem2.6eq1}
\sum_{i\not =j}|A_0^i\cap A_0^j||A_1^i\cap A_1^j|=|\{(i,j,a,b):i\not =j, a\in A_0^i\cap A_0^j, b\in A_1^i\cap A_1^j\}|=\sum_{a,b}2\binom{s_{a,b}}{2}.
\end{equation}

On the other hand, 
\begin{equation}
\label{lem2.6eq2}
\sum_{a,b}s_{a,b}=|\{(i,a,b): a\in A^i_0, b\in A^i_1\}|=\sum_i |A_0^i||A_1^i|\geq \frac{1}{5}\sum_i|A^i|^2\geq \frac{M^2}{5n}.
\end{equation}

(Here we used that for all $i$, $|A^i|>1$ if $C$ is large enough.)
Hence, by \eqref{lem2.6eq1}, \eqref{lem2.6eq2}, and Jensen's inequality, if $C$ is large enough, we get 
\begin{align*}
    \sum_{i\not =j}|A_0^i\cap A_0^j||A_1^i\cap A_1^j|&=\sum_{a,b}2\binom{s_{a,b}}2\\
    &\geq2n'^2\binom{\sum_{a,b}s_{a,b}/n'^2}{2}\\
    &\geq 2n'^2\binom{M^2/(5nn'^2)}{2}\\
    &\geq \frac{M^4}{50n^2n'^2}\ge\frac{M^4}{50K^2n^4}.
\end{align*}
This completes the proof of the lemma.
\end{proof}

\subsubsection{Proof of the main result}

We are now ready to put everything together to complete the proof of Theorem~\ref{theorem_almostregular} (and hence also the proof of Theorem~\ref{theorem_linear orders}).

\begin{proof}[Proof of Theorem~\ref{theorem_almostregular}]
    Let $G$ denote the graph corresponding to the linear orders $A^1, \dots, A^n$ and the set of $n'$ symbols.  Let $M$ denote the number of edges in $G$. We need to show that $M=O(n^{3/2})$. Suppose for a contradiction that $M>Cn^{3/2}$, where $C$ is a constant to be chosen later.
    
    By Lemma~\ref{lemma_gainbound}, Lemma~\ref{lemma_cherrycounting} and~\eqref{equation_Supperbound}, if $C$ is sufficiently large, we have
    \begin{align*}
        S\leq \sum_{i\not =j}|A^i\cap A^j|-\frac{M^4}{25K^2n^4}.
    \end{align*}
    Observe that $M/n$ and $M/n'$ are the average degrees of the vertices in the two parts of $G$. Thus, both $M/n$ and $M/n'$ are in between the minimum degree $\delta=\delta(G)$ and the maximum degree $\Delta=\Delta(G)$ of the graph $G$. For each symbol $a$, let $d_a$ denote its degree in $G$, that is, the number of sequences $A^i$ containing it. We have
    \begin{align*}
        \sum_{i\not =j}|A^i\cap A^j|&=|\{(i,j,a):i\not =j,a\in A^i\cap A^j\}|=\sum_a 2\binom{d_a}{2}
        \leq n'\Delta^2\leq n'K^2\delta^2
        \le \frac{K^2M^2}{n},
    \end{align*}
    giving
    \begin{align*}
        S\leq \frac{K^2 M^2}{n}-\frac{M^4}{25K^2n^4}.
    \end{align*}
    But Lemma~\ref{lemma_Slowerbound} gives $S\geq -\frac{1}{2}\sum |A^i|^2\geq -\frac{1}{2}n\Delta^2\geq -\frac{1}{2}K^2n\delta^2\geq -\frac{K^2M^2}{2n}$. Therefore, $\frac{M^4}{25K^2n^4}\leq \frac{3}{2}\frac{K^2M^2}{n}$, i.e.,
        $M\leq \sqrt{\frac{75}{2}}K^2n^{3/2}$. Choosing $C>7K^2$ and so that it is also large enough for Lemma~\ref{lemma_cherrycounting} to apply gives the desired contradiction.
\end{proof}

\section{Edge-ordered graphs} \label{sec:edge-ordered graphs}

\subsection{Edge-ordered four-cycles}

In this subsection we prove Theorem~\ref{thm:ordered C4}, giving a tight bound on the extremal number of edge-ordered $4$-cycles.

\begin{proof}[Proof of Theorem \ref{thm:ordered C4}]
    As mentioned in the introduction, the lower bound follows trivially from known bounds (see e.g., \cite{furedi2013history}) on the extremal number of the unordered $4$-cycle.
    
    For the upper bound, we will prove that $\ex_<(n,C_4^{1243})\leq (C/2)n^{3/2}$, where $C$ is the constant from Theorem \ref{theorem_linear orders}. Let $G$ be an edge-ordered graph on $n$ vertices with $e(G)>(C/2)n^{3/2}$, where the edges of $G$ are ordered by the linear order $<_G$. For each vertex $u$ in $G$, define a linear order $A^u$ on the set of neighbors of $u$ in which $x<y$ if $ux <_G uy$ in the edge-ordering of $G$. Note that $\sum_u|A^u|=2e(G)>Cn^{3/2}$, so by Theorem \ref{theorem_linear orders}  we obtain distinct vertices $u,v$ of $G$ and distinct vertices $x_1,x_2,x_3$ in the common neighborhood of $u$ and $v$ such that $x_1$, $x_2$ and $x_3$ appear in the same order in both $A^u$ and $A^v$. Without loss of generality, this means that $ux_1 <_G ux_2 <_G ux_3$ and $vx_1 <_G vx_2 <_G vx_3$. By the pigeonhole principle, there exist $1\leq i<j\leq 3$ such that either we have both $ux_i>_G vx_i$ and $ux_j >_G vx_j$ or we have both $ux_i <_G vx_i$ and $ux_j <_G vx_j$. In either case, it is easy to check that the four-cycle $ux_ivx_j$ is isomorphic to $C_4^{1243}$, so $G$ contains $C_4^{1243}$. Hence, we have $\ex_<(n,C_4^{1243})\leq(C/2)n^{3/2}$, as desired.
\end{proof}

\subsection{Edge-ordered forests}

In this subsection we prove Theorem \ref{thm:tree construction}. The proof relies on a very recent construction of Pettie and Tardos~\cite{pettie2024refutation} giving new lower bounds for the extremal numbers of some zero-one matrices. The setting of zero-one matrices is essentially equivalent to the one of vertex-ordered graphs mentioned in the introduction (see, e.g.,~\cite{PT06}), so to simplify this presentation we will only refer to zero-one matrices here instead of vertex-ordered graphs. For zero-one matrices $A$ and $M$, we say that $M$ \textit{contains} $A$ if, by deleting some rows and columns from $M$, and possibly turning some of its $1$-entries to $0$-entries, we can obtain $A$. Otherwise, we say that $M$ \emph{avoids} $A$. The \textit{weight} of a zero-one matrix $M$, denoted $w(M)$, is the number of $1$-entries in $M$. The \textit{extremal number} of a zero-one matrix $A$, denoted $\Exx(n,A)$, is the maximum possible weight of an $n\times n$ zero-one matrix that avoids $A$.

Let $A$ be a zero-one matrix. We turn $A$ into an edge-ordered graph $G(A)$ by taking the bipartite graph whose adjacency matrix is $A$ and ordering its edges lexicographically. Note that $G(A)$ has vertices corresponding to the rows and columns of $A$ and each edge connects a row-vertex and a column-vertex. By lexicographic ordering, we mean that the edges are ordered first according to their column-vertex (left to right) and the edges incident to the same column-vertex are ordered according to their row-vertex (top to bottom). 

We will need the following lemma which connects the extremal functions of $A$ and $G(A)$. Let us call the following matrix the \emph{hat} pattern:
$$\left(\begin{array}{ccc}0&1&0\\1&0&1\end{array}\right).$$

\begin{lemma}\label{connect}
Let $A$ be a zero-one matrix without an all-0 column. If $A$ contains the hat pattern and each pair of consecutive rows in $A$ share a $1$-entry in a common column, then
$$\ex_<(2n,G(A))\ge\Exx(n,A).$$
\end{lemma}

\begin{proof}
    Let $M$ be an $n\times n$ zero-one matrix avoiding $A$ and having the maximum possible weight i.e., $w(M)=\Exx(n,A)$. Notice that the edge-ordered graph $G(M)$ has $2n$ vertices and $\Exx(n,A)$ edges, so it is enough to prove that it avoids $G(A)$. Assume for a contradiction that it contains $G(A)$ as an edge-ordered subgraph. Let $f$ be an isomorphism that maps $G(A)$ to an edge-ordered subgraph of $G(M)$.
    
    Both $G(A)$ and $G(M)$ are bipartite graphs with their edges connecting row-vertices to column-vertices. By assumption, each pair of consecutive row-vertices in $G(A)$ are connected by a path of length two and it has no isolated column-vertices, so $G(A)$ is a connected graph. Thus, $f$ must either map all row-vertices of $G(A)$ to row-vertices of $G(M)$ and all column-vertices of $G(A)$ to column-vertices of $G(M)$, or else it maps all row-vertices to column-vertices and all column-vertices to row-vertices.

    First, we show that the latter of the two is not possible. Indeed, let $e_1<e_2<e_3$ be the three edges of $G(A)$ that correspond to a hat pattern. Note that $e_1$ and $e_3$ share a common row-vertex $v$ and $e_2$ is not incident to $v$. Now, if $f$ maps $v$ to a column vertex in $G(M)$, then the image of $e_1$ and the image of $e_3$ would share this column-vertex but the image of $e_2$ would not, making $f(e_1)<f(e_2)<f(e_3)$ impossible, contradicting our assumption that $f$ is an isomorphism of edge-ordered graphs.

    Therefore, $f$ maps all row-vertices to row-vertices and all column-vertices to column-vertices. We claim that $f$ preserves the order of both the row-vertices and the column-vertices. For two column-vertices $u, v$ of $G(A)$, consider edges $e, e'$ of $G(A)$, such that $e$ is incident to $u$ and $e'$ is incident to $v$. (Here the existence of $e$ and $e'$ is guaranteed by our assumption that $A$ contains no all-$0$ column.) If $f$ switched the order of $u$ and $v$, it would also switch the order of the incident edges $e$ and $e'$, a contradiction to our assumption that $f$ is an isomorphism of edge-ordered graphs. To show that $f$ preserves the order of row-vertices, consider two row-vertices $u, v$ of $G(A)$ corresponding to consecutive rows of $A$ and the edges $h_1, h_2$ of $G(A)$ in a path of length two connecting $u$ and $v$ (where the existence of such a path is guaranteed by our assumption that a column exists with a $1$-entry in both rows). Again, if $f$ switched the order of $u, v$, it would also switch the order of the two edges $h_1, h_2$ in this path, a contradiction.
    
    As $f$ preserves the order of both row-vertices and column-vertices, the image of $f$ provides a submatrix of $M$ which is isomorphic to $A$. This is a contradiction again, as $M$ is assumed to avoid $A$. This proves the lemma.
\end{proof}

Let us comment on the several assumptions made in Lemma~\ref{connect}. The assumptions that $A$ contains no all-0 column but contains the hat pattern are technical in nature and can be relaxed. However, the assumption that each pair of consecutive rows in $A$ have a $1$-entry in a common column is essential as shown by the following example. Let us define zero-one matrices $Z$ and $Z'$ as follows.

\begin{equation*}Z=\left(
\begin{array}{cccc}
    1 & 0 & 1 & 0\\
    0 & 1 & 0 & 1\\
    1 & 0 & 0 & 1
\end{array}\right)\;\;\;\;
Z'=\left(
\begin{array}{cccc}
    0 & 1 & 0 & 1\\
    1 & 0 & 1 & 0\\
    1 & 0 & 0 & 1
\end{array}\right)
\end{equation*}

Observe that $G(Z)$ and $G(Z')$ are isomorphic edge-ordered graphs. Thus, if a matrix $A$ avoids $Z$ but not $Z'$, then $G(A)$ \emph{does not} avoid $G(Z)$ and we cannot prove the inequality stated in Lemma~\ref{connect} for $A=Z$ with the method above. Indeed, $\ex_<(2n,G(Z))=n2^{O(\sqrt{\log n})}$ follows from the general upper bound in \cite{kucheriya2023characterization} but the best upper bound known for $\Exx(n,Z)$ is somewhat weaker: $\Exx(n,Z)=n2^{O((\log n)^{2/3})}$, see~\cite{KTTW}.

To prove Theorem~\ref{thm:tree construction}, we will also need the simple observation that the order-chromatic number of $G(A)$ is 2 for all (not-all-0) zero-one matrices $A$. In fact, this is a characterization of edge-ordered graphs with order-chromatic number 2, see part (3) of Corollary~2.6 from \cite{GMNPTV23}.

Let $S_t$ denote the $3\times(2t)$ zero-one matrix in which the first row consists of alternating 0's and 1's starting with a 1, the last row also consists of alternating 0's and 1's, but this time starting with a 0, while the middle row contains exactly two 1's and they are in the first and last columns. Here are two examples:

\begin{equation*}S_2=\left(
\begin{array}{cccc}
    1 & 0 & 1 & 0\\
    1 & 0 & 0 & 1\\
    0 & 1 & 0 & 1
\end{array}\right)\;\;\;\;
S_4=\left(
\begin{array}{cccccccc}
    1 & 0 & 1 & 0 & 1 & 0 & 1 & 0\\
    1 & 0 & 0 & 0 & 0 & 0 & 0 & 1\\
    0 & 1 & 0 & 1 & 0 & 1 & 0 & 1
\end{array}\right)
\end{equation*}

In the next theorem, $\log$ denotes the binary logarithm.

\begin{theorem}[Pettie and Tardos~{\cite[Theorem~2.2]{pettie2024refutation}}]\label{theorem_pettietardos}
For any fixed $t\ge2$, we have $\Exx(n,S_t)\geq n2^{(1-o(1))\sqrt{\log t\log n}}$.
\end{theorem}

We are now ready to prove Theorem~\ref{thm:tree construction} using Lemma~\ref{connect} and Theorem~\ref{theorem_pettietardos}.

\begin{proof}[Proof of Theorem~\ref{thm:tree construction}]
Note that $S_t$ satisfies all conditions of Lemma~\ref{connect} if $t\ge2$. Thus,
$$\ex_<(2n,G(S_t))\ge\Exx(n,S_t)\ge n2^{(1-o(1))\sqrt{\log t\log n}}.$$
    Here $G(S_t)$ is a tree. In fact, it can be obtained from two disjoint $t$-edge stars by connecting a new vertex to one leaf of each star. So $G(S_t)$ is an edge-ordered tree of order chromatic number two. As $t$ is arbitrary, the constant factor $\sqrt{\log t}$ in the exponent can be made arbitrarily large.
    \end{proof}

    Note that the special case of $t=2$ gives the edge-ordered graph $G(S_2)$ which is the path $P_6^{412563}$, i.e., the six-edge path with the edge-ordering given by labeling the edges along the path using the labels $4,1,2,5,6,3$ and ordering the edges according to their label. This simple example already refutes the conjecture of Kucheriya and Tardos~\cite{kucheriya2023characterization} mentioned in the introduction, whereas Theorem~\ref{thm:tree construction} constructs a family of counterexamples showing that even the constant factor $C$ in the exponent be made arbitrarily large.
    
\section{Concluding remarks} \label{sec:concluding remarks}

\subsection{Geometric graphs without self-intersecting four-cycles}

We used our Theorem \ref{thm:main cyclic} to prove that every $n$-vertex topological graph without a self-intersecting four-cycle has $O(n^{3/2})$ edges. Here we give a very short proof of the following weaker result, concerning geometric graphs, which only uses our bound on the extremal number of the edge-ordered four-cycle $C_4^{1243}$. We do so to further illustrate the various connections between the several topics studied in our paper.

\begin{theorem}
    If $G$ is an $n$-vertex geometric graph which contains no self-intersecting $4$-cycles, then $e(G)=O(n^{3/2})$.
\end{theorem}
\begin{proof}
    By taking a rotation if necessary, we may assume, without loss of generality, that our geometric graph is embedded into $\mathbb{R}^2$ such that no edge is vertical. Independently, and uniformly at random, partition the vertices of $G$ into two classes: $L$ (`left vertices') and $R$ (`right vertices'). Let $G'$ be the (bipartite) subgraph of $G$ obtained by keeping only the edges which are of the form $\ell r$ with $\ell\in L, r\in R$ and $\ell$ having smaller first coordinate than $r$. Note that $\mathbb{E}(e(G'))=\frac{1}{4}e(G)$, so it suffices to show that we always have $e(G')=O(n^{3/2})$.

    We turn $G'$ into an edge-ordered graph by ordering its edges according to slope (and ordering edges with identical slope arbitrarily). That is, edges with larger slope are larger in the edge-order. One can check that in any non-self-intersecting four-cycle in $G'$, the edges of smallest and largest slope must be adjacent. Hence, $G'$ cannot contain a copy of the edge-ordered four-cycle $C_4^{1243}$. The result then follows from Theorem~\ref{thm:ordered C4}.
\end{proof}

\subsection{Linear orders with restricted intersections}

Our Theorem \ref{theorem_linear orders} naturally leads to the following question.

\begin{question}\label{qn:linear orders}
    Is it true that for every positive integer $k$ there exists $C=C(k)$ such that for any positive integer~$n$, if $A^1,\dots,A^n$ are linear orders on some subsets of a set of $n$ symbols such that no $k$ symbols appear in the same order in any two distinct linear orders $A^i$ and $A^j$, then $\sum_{i=1}^n |A^i|\leq Cn^{3/2}$?
\end{question}

Given the many applications of Theorem \ref{theorem_linear orders}, it is very likely that an affirmative answer to the above question would have several further important consequences.

\bigskip

\noindent \textbf{Acknowledgements.} We thank David Conlon, António Girão, Gaurav Kucheriya and Van Magnan for interesting discussions related to the topic of this paper.

\end{document}